\pdfoutput=1
\RequirePackage{ifpdf}
\ifpdf 
\documentclass[pdftex]{sigma}
\else
\documentclass{sigma}
\fi

\numberwithin{equation}{section}

\newtheorem{Theorem}{Theorem}[section]
\newtheorem*{Theorem*}{Theorem}
\newtheorem{Corollary}[Theorem]{Corollary}

\newtheorem{Proposition}[Theorem]{Proposition}
 { \theoremstyle{definition}

\newtheorem{Remark}[Theorem]{Remark} }

\newcommand*{\R}{\mathbb{R}}
\newcommand*{\C}{\mathbb{C}}
\newcommand*{\Q}{\mathbb{Q}}
\renewcommand*{\H}{\mathbb{H}}
\newcommand*{\Z}{\mathbb{Z}}
\newcommand*{\N}{\mathbb{N}}

\DeclareMathOperator{\PSL}{PSL}

\begin{document}

\allowdisplaybreaks

\newcommand{\arXivNumber}{2308.06158}

\renewcommand{\PaperNumber}{053}

\FirstPageHeading

\ShortArticleName{Infinitesimal Modular Group: $q$-Deformed $\mathfrak{sl}_2$ and Witt Algebra}

\ArticleName{Infinitesimal Modular Group:\\ $\boldsymbol{q}$-Deformed $\boldsymbol{\mathfrak{sl}_2}$ and Witt Algebra}

\Author{Alexander THOMAS}

\AuthorNameForHeading{A.~Thomas}

\Address{Universit\"at Heidelberg, Berliner Str. 41-49, 69120 Heidelberg, Germany}
\Email{\href{mailto:athomas@mathi.uni-heidelberg.de}{athomas@mathi.uni-heidelberg.de}}
\URLaddress{\url{https://thomas-math.wixsite.com/maths/en}}

\ArticleDates{Received December 01, 2023, in final form June 03, 2024; Published online June 20, 2024}

\Abstract{We describe new $q$-deformations of the 3-dimensional Heisenberg algebra, the simple Lie algebra $\mathfrak{sl}_2$ and the Witt algebra. They are constructed through a realization as differential operators. These operators are related to the modular group and $q$-deformed rational numbers defined by Morier-Genoud and Ovsienko and lead to $q$-deformed M\"obius transformations acting on the hyperbolic plane.}


\Classification{35A01; 65L10; 65L12; 65L20; 65L70}

\section{Introduction and results}

The construction of $q$-deformed rational numbers by Morier-Genoud and Ovsienko \cite{morier2020continued} starts from the observation that rational numbers are generated by the image of zero under the action of the modular group $\PSL_2(\Z)$. This group is generated by the translation $T(x)=x+1$ and the inversion $S(x)=-1/x$. The only relations between these operations are $S^2=\mathrm{id}=(ST)^3$.

The $q$-deformed integers $[n]_q=1+q+q^2+\dots +q^{n-1}=\frac{1-q^n}{1-q}$, where $q\in \C^*$, satisfy the relation $[n+1]_q=q[n]_q+1$. It is natural to introduce as $q$-analog to the translation $T$ the transformation $T_q(x)=qx+1$. The map $S_q(x)=-1/(qx)$ satisfies $S_q^2=\mathrm{id}=(S_qT_q)^3$. The~\emph{$q$-rational numbers} are then defined by the image of zero under the action by $T_q$ and $S_q$ using for example the continued fraction representation of a rational number. Since these operations are M\"obius transformations, we can represent them in matrix form as follows:
\[
T_q=\begin{pmatrix} q&1\\ 0 & 1 \end{pmatrix}\qquad \text{and}\qquad S_q=\begin{pmatrix} 0&-1\\ q & 0 \end{pmatrix}.
\]
This coincides with the \emph{reduced Burau representation} of the braid group $B_3$ with parameter~${t=-q}$~\cite{burau}. Indeed the standard generators of $B_3$ are represented by $\sigma_1=T_q$ and $\sigma_2=S_qT_qS_q=\left(\begin{smallmatrix}1 & 0 \\ -q & q\end{smallmatrix}\right)$. One easily checks the braid relation $\sigma_1\sigma_2\sigma_1=\sigma_2\sigma_1\sigma_2$. The faithfulness of specializations of the Burau representation (where $q$ is not a formal parameter, but a non-zero complex number) is an open question \cite[Section 7]{bharathram2021burau}. It was studied for real values of $q$ in \cite{scherich2020classification}. In~\cite{morier2023}, a link to $q$-deformed rational numbers allows to partially solve the open question.

Using Taylor expansions of $q$-rational numbers, one can define $q$-real numbers \cite{morier2022q} which are power series in $q$ with integer coefficients. A natural question is \emph{how to do analysis with these~$q$-real numbers?} Basic functions on real numbers are monomials and the exponential function, which are eigenfunctions of the vector fields associated to~$\mathfrak{sl}_2(\R)$ \big(acting on the completed line~$\mathbb{RP}^1$\big). The goal of our investigation is to $q$-deform these vector fields and to analyze their eigenfunctions.

Following a suggestion of Valentin Ovsienko, we can associate to $T_q$ a differential operator $D_{-1}(q)$, which corresponds to the infinitesimal $q$-shift. For $q = 1$, we have ${D_{-1}(1)=\partial={\rm d}/{\rm d}x}$. This operator is given by
\[
D_{-1}:=(1+(q-1)x)\partial.
\]

One can directly check that $D_{-1}$ commutes with $T_q$, where $T_q$ acts on the space of functions by precomposition.
The starting point of the paper is the question whether there is a differential operator associated to $S_q$. This would allow to define in some sense a Lie algebra for the modular group $\PSL_2(\Z)$, or an infinitesimal version of the Burau representation of $B_3$.

In the classical setting for $q=1$, there is an operator which \emph{anti-commutes} with $S$:
\[
S\circ x\partial+x\partial\circ S = 0,
\]
where $S$ acts on the space of functions by precomposition.
We introduce the differential operator~$D_0$, a $q$-deformed version of $x\partial$, given by
\[D_0:=(1+(x-1)q)D_{-1}=(1+(x-1)q)(1+(q-1)x)\partial.
\]

We will see that $D_0$ anti-commutes with $S_q$. Together with $D_1:=S_q\circ D_{-1}\circ S_q$, a deformation of $x^2\partial$, we get three differential operators which are closed under the bracket (see Theorem \ref{Prop:deformed-sl2}):
\begin{Theorem}
The operators $D_{-1}$, $D_0$ and $D_1$ form a Lie algebra with brackets
\begin{align*}
&[D_0,D_1]= \big(q^2-q+1\big)D_1+(1-q)D_0, \qquad
[D_0,D_{-1}] = -\big(q^2-q+1\big)D_{-1}+(1-q)D_0, \\
&[D_{-1},D_1]= 2D_0+(1-q)(D_{1}-D_{-1}).
\end{align*}
\end{Theorem}
The theorem tells us that the module over $\R[q]$ generated by $D_{-1}$, $D_0$ and $D_1$ is a deformation of the Lie algebra $\mathfrak{sl}_2(\R)$ which we recover for $q=1$. The Lie algebra $\mathfrak{sl}_2$ being simple, it does not allow for non-trivial deformations. Hence our deformation is isomorphic to $\mathfrak{sl}_2$ as a Lie algebra, but they are different as $\mathbb{Z}[q]$-modules. This is similar to quantum groups.

A fundamental role is played by the M\"obius transformation
\[
g_q(x)=\frac{1+(x-1)q}{1+(q-1)x},
\]
which is a deformation of the identity. It is the eigenfunction of $D_0$ with eigenvalue $q^2-q+1$ and normalization $g_q(0)=1-q$.
We call it the \emph{$q$-rational transition map} since it makes a passage between two different $q$-deformations of rational numbers studied in \cite{bapat2023q}. More precisely (see Theorem \ref{Thm:passage}):
\begin{Theorem}
The two $q$-deformations of rational numbers defined in {\rm \cite[Definition 2.6]{bapat2023q}} are linked via
\[
g_q\left(\left[\frac{r}{s}\right]_q^\sharp\right)=\left[\frac{r}{s}\right]_{q^{-1}}^\flat.\]
\end{Theorem}

This theorem comes from the interplay between $g_q, T_q$ and $S_q$ given by $g_q\circ T_q=T_{q^{-1}}\circ g_q$ and $g_q\circ S_q=S_{q^{-1}}\circ g_q$ (see Proposition \ref{Prop:qg-Tq-Sq}). The $q$-rational transition map also satisfies a sort of duality between $q$ and $x$:
\[g_q(x)g_x(q)=1.
\]

The map $g_q$, as well as its multiplicative inverse $g_q^{-1}=1/g_q$, behave very well with the three operators $D_{-1}$, $D_0$ and $D_1$ (see Propositions \ref{Prop:op-diff-gq} and \ref{Prop:mult-by-gq}):

\begin{Proposition}
The $q$-rational transition map $g_q$ and the differential operators $D_{-1}$, $D_0$, and~$D_1$ interact in the following way:
\begin{enumerate}\itemsep=0pt
\item[$(1)$] $D_0(g_q) = \big(q^2-q+1\big)g_q$, $D_{-1}(g_q) = q+(1-q)g_q$, $D_1(g_q) = (q-1)g_q+g_q^2$,
\item[$(2)$] $g_qD_0 = (1-q)D_0+\big(q^2-q+1\big)D_1$, $g_qD_{-1}= D_0+(1-q)D_1$,
\item[$(3)$] $qg_q^{-1}D_0 = (q-1)D_0+\big(q^2-q+1\big)D_{-1}$, $qg_q^{-1}D_{1}= D_0+(q-1)D_{-1}$.
\end{enumerate}
\end{Proposition}

These relations allow a deformation of the Witt algebra, the complexification of the Lie algebra of polynomial vector fields on the circle. The Witt algebra is described by a vector space basis $(\ell_n)_{n\in\Z}$ with bracket given by
\[
[\ell_n,\ell_m]=(m-n)\ell_{n+m}.
\]
This algebra can be realized as differential operators (or equivalently as vector fields) via ${\ell_n=x^{n-1}\partial}$.
Putting for $n>1$:
\[
D_n = g_q^{n-1}D_1 \qquad \text{and}\qquad D_{-n} = \big(qg_q^{-1}\big)^{n-1}D_{-1},
\]
we get a deformation of the Witt algebra (see Theorem \ref{Prop-deformed-Witt}):
\begin{Theorem}
The $(D_n)_{n\in\mathbb{Z}}$ form a Lie algebra with bracket given by $($where $n,r >0)$:
\begin{align*}
&[D_0,D_n]= n\big(q^2-q+1\big)D_n+\big(q^2-q+1\big)\displaystyle\sum_{k=1}^{n-1}(1-q)^kD_{n-k} + (1-q)^nD_0, \\
&[D_n,D_{n+r}]=rD_{2n+r}+(q-1)rD_{2n+r-1},\\
&[D_{-n},D_n]= 2nq^{n-1}D_0+(2n-1)q^{n-1}(q-1)(D_{-1}-D_1),\\
&[D_{n+r},D_{-n}]= (q-1)q^{n-1}(2n+r-1)D_{r+1}-\big(q^2+(2n+r-2)q+1\big)q^{n-1}D_r,\\
&\phantom{[D_{n+r},D_{-n}]=}{} -q^{n-1}\big(q^2-q+1\big)\displaystyle\sum_{k=1}^{r-1}(1-q)^kD_{r-k}-(1-q)^rq^{n-1}D_0.
\end{align*}
The remaining brackets $[D_0,D_{-n}]$, $[D_{-n},D_{-n-r}]$ and $[D_n,D_{-n-r}]$ obey similar formulas.
\end{Theorem}

Integrating the vector fields associated to $D_{-1}$, $D_0$ and $D_1$ on the hyperbolic plane, we get M\"obius transformations. We speculate about a $q$-deformed hyperbolic plane on which these transformations naturally act. The boundary of this deformed hyperbolic plane should be the~$q$-deformed real numbers. Other interesting open questions include the link between our~$q$-deformed $\mathfrak{sl}_2$ and the quantum group $\mathcal{U}_q(\mathfrak{sl}_2)$, or the existence of a central extension of our deformed Witt algebra, which would give a deformed Virasoro algebra.

Deformations of rational numbers were introduced in \cite{morier2020continued}, extended to real numbers in \cite{morier2022q} and to Gaussian integers in \cite{ovsienko2021towards}. Many different deformations of the Witt algebra or its central extension, the Virasoro algebra, have been introduced in the past: first in \cite{curtright1990deforming} and then in \cite{chaichian1991q} deforming the matrix Lie bracket to $[A,B]_q=qAB-q^{-1}BA$. This also deforms the Jacobi identity. A similar construction was done in \cite{hu2005q} viewing the Witt algebra as space of derivations of $\C\big[x^{\pm 1}\big]$ and using the $q$-differential $\partial_q(f)=\tfrac{f(qx)-f(x)}{qx-x}$. This was generalized in~\cite{hartwig2006deformations} to more general $\sigma$-derivatives. Deforming the cocycle gives a $q$-Virasoro algebra in \cite{kassel1992cyclic}, developed into a~theory of $q$-deformed pseudo-differential operators in \cite{khesin1997extensions}. A deformation as Lie algebra in terms of an operator product expansion is given in \cite{shiraishi1996quantum}. A similar proposal can be found in \cite[equation (1.3)]{frenkel1996quantum}, using a $q$-deformed Miura transformation. In \cite[equation~(38)]{nedelin2017q}, the deformation
\[
[T_m(q),T_n(q)]=([-n]_q-[-m]_q)(T_{n+m}\big(q^2\big)-T_{n+m}(q))\]
 is studied. Yet another proposal from \cite[formula~(3.18)]{nigro2016q} gives operators $D_n(q)$ for $n\in\Z$ with commutator~${[D_n(q),D_m(q)]=\big(q-q^{-1}\big)[n-m]_qD_{n+m}\big(q^2\big)}$ (removing the central extension). Finally, in~\cite{avan2017deformed}, a two-dimensional deformation using elliptic algebras is studied. All these approaches are different from ours.

\textbf{Structure of the paper.} In Section \ref{Sec:sl2}, we introduce and study the deformation of $\mathfrak{sl}_2$, the Heisenberg algebra and the $q$-rational transition map. This is broadened in Section \ref{Sec:Witt} to a~deformed Witt algebra. In the final Section \ref{Sec:Moebius}, we study the M\"obius transformations associated to these deformations.

\section[Deformed sl\_2 and Heisenberg algebra]{Deformed $\boldsymbol{\mathfrak{sl}_2}$ and Heisenberg algebra}\label{Sec:sl2}

The group $\mathrm{SL}_2$ acts naturally on the projective line $\mathbb{P}^1$. We will work over $\R$ or $\C$. Differentiating this action at the identity gives a realization of the Lie algebra $\mathfrak{sl}_2$ as vector fields on $\mathbb{P}^1$.
Using the two standard charts of $\mathbb{P}^1$ with transition function $x\mapsto 1/x$, the image of $\mathfrak{sl}_2\to \mathrm{Vect}\big(\mathbb{P}^1\big)$ is generated by $\partial$, $x\partial$ and $x^2\partial$ written in the first chart, where we use the notation $\partial={\rm d}/{\rm d}x$. One readily checks that these expressions are well-defined over the second chart.

We construct a deformation of these three differential operators. They come as a realization of a Lie algebra which itself deforms $\mathfrak{sl}_2$. Together with a $q$-deformed identity map, we deform the 3-dimensional Heisenberg algebra.

\subsection[Deformed sl\_2]{Deformed $\boldsymbol{\mathfrak{sl}_2}$}
On $\mathbb{P}^1$, consider the M\"obius transformations
\[
T_q(x)=qx+1 \qquad \text{and}\qquad S_q(x)=-\frac{1}{qx},
\]
where $q\in \C^*$ is fixed or seen as a formal parameter.
They deform the translation $x\mapsto x+1$ and the inversion $x\mapsto -1/x$. These transformations act on the space of functions on $\mathbb{P}^1$ by precomposition.

Consider the differential operator $D_{-1}$ on $\mathbb{P}^1$ which is defined in the first chart by
\[
D_{-1}:=(1+(q-1)x)\partial.
\]

\begin{Proposition}\label{Prop:commute-T_q}
The operators $D_{-1}$ and $T_q$ commute, where $T_q$ acts on the space of functions by precomposition.
\end{Proposition}
\begin{proof}
For a function $f(x)$, we have on the one side
\[
D_{-1}\circ T_q(f(x))=D_{-1}(f(qx+1)) = (1+(q-1)x)qf'(qx+1).
\]
On the other side,
\[
T_q\circ D_{-1}(f(x))=T_q((1+(q-1)x)f'(x)) = (1+(q-1)(qx+1))f'(qx+1).
\]
Both expressions coincide.
\end{proof}

The unique eigenfunction $E_q$ of $D_{-1}$ with eigenvalue 1 and normalization $E_q(0)=1$ is a~$q$-deformation of the exponential function, called the \emph{Tsallis exponential} \cite{tsallis1988possible}. This was first observed by Valentin Ovsienko and Emmanuel Pedon.\footnote{Unpublished, private communication.} To find $E_q$, one has to solve ${f=D_{-1}f=(1+(q-1)x)f'}$, i.e., $(\ln f)'=\frac{1}{1+(q-1)x}$. The solution is given by
\[
E_q(x)=(1+(q-1)x)^{\frac{1}{q-1}}.
\]
It satisfies $E_q(qx+1)=E_q(1)E_q(x)$ since $E_q(qx+1)=T_qE_q$ is also an eigenfunction of $D_{-1}$ with eigenvalue 1.

The main new operator we introduce is the following:
\[
D_0:=(1+(x-1)q)D_{-1}=(1+(x-1)q)(1+(q-1)x)\partial.
\]

\begin{Proposition}\label{Prop:commute-S_q}
The operators $D_0$ and $S_q$ anti-commute, where $S_q$ acts on the space of functions by precomposition.
\end{Proposition}
The proof is a direct verification, similar to the proof of Proposition \ref{Prop:commute-T_q}. An equivalent statement is $S_q\circ D_0\circ S_q=-D_0$.

\begin{proof} For a function $f(x)$, we have on the one side
\[
D_0\circ S_q(f(x))=D_0 f\left(-\frac{1}{qx}\right)=(1+(x-1)q)(1+(q-1)x)f'\left(-\frac{1}{qx}\right)\frac{1}{qx^2}.
\]
On the other hand,
\begin{align*}
S_q\circ D_0(f(x)) &= S_q((1+(x-1)q)(1+(q-1)x)f'(x)) \\
&= \left(1+q\left(-\frac{1}{qx}-1\right)\right)\left(1-\frac{1}{qx}(q-1)\right)f'\left(-\frac{1}{qx}\right)\\
&= -\frac{1}{qx^2}(1+(x-1)q)(1+(q-1)x)f'\left(-\frac{1}{qx}\right).\tag*{\qed}
\end{align*}\renewcommand{\qed}{}
\end{proof}

More generally, we can find all operators $D$ of the form $p(x)\partial$ which anti-commute with $S_q$. The relation $\{D,S_q\}=0$ gives
\[
p(x)=-qx^2p\left(-\frac{1}{qx}\right).
\]
Adding as constraint that $p$ has to be polynomial, it is clear that it is of degree at most 2. Plugging in $p(x)=p_0+p_1x+p_2x^2$ gives a solution for any $p_1$ and $p_2=-qp_0$. In other words, the two fundamental solutions are $p(x)=x$ and $p(x)=1-qx^2$.
Note in particular that the undeformed operator $x\partial$ still anticommutes with $S_q$.
The particular choice above for $D_0$ is~${p_1=-1+3q-q^2}$ and $p_0=1-q$. We will see below why this is the simplest choice.

Let us determine the eigenfunctions of $D_0$ with eigenvalue $\alpha$. One has to solve $\alpha f=D_0 f$, i.e., $(\ln f)'=\frac{\alpha}{(1+(q-1)x)(1+(x-1)q)}$. The solutions are
\[
\left(\frac{1+(x-1)q}{1+(q-1)x}\right)^{\frac{\alpha}{q^2-q+1}}.
\]
We define the \emph{$q$-rational transition map}
\begin{equation}\label{Eq-def-g_q}
g_q(x)=\frac{1+(x-1)q}{1+(q-1)x},
\end{equation}
which is the unique eigenfunction of $D_0$ with eigenvalue $q^2-q+1$ and normalization $g_q(0)=1-q$. We can think of $g_q$ as a deformation of the identity map. We study this function more in detail below in Section~\ref{Sec:Ovsienko}.

Now we come back to the discussion why our $D_0$ is the simplest choice. Consider an operator~${D=p(x)\partial}$ anti-commuting with $S_q$, i.e., of the form $p(x)=p_0+p_1x-qp_0x^2$ with arbitrary~${p_0,p_1\in\Z[q]}$. We impose that $D$ deforms $x\partial$, that is $p_0(1)=0$ and $p_1(1)=1$. We also impose the leading terms of $p_0$, $p_1$ to be $\pm 1$. We wish that the eigenfunctions of $D$ are M\"obius transformations in $\Z[q]$. This is only the case if the discriminant of $p_0+p_1x-qp_0x^2$ is a square in~$\Z[q]$. This leads to the equation $p_1(q)^2+4qp_0(q)^2=R(q)^2$ for some $R\in\Z[q]$. This is equivalent to $4qp_0^2=(R-p_1)(R+p_1)$. Excluding the case where $p_0=0$ which leads to the undeformed operator $x\partial$, the next simplest case is $p_0(q)=1-q$. By treating all possible factorizations of $4q(1-q)^2$, we see that the $p_1$ with lowest degree has to be $p_1(q)=-1+3q-q^2$ which is the case for our choice $D_0$.

We complete the operators $D_{-1}$ and $D_0$ to a deformed $\mathfrak{sl}_2$. For that, we wish to deform $x^2\partial$.
Note that $x^2\partial = S\circ \partial\circ S$. This motivates the following definition:
\[
D_1:=S_q\circ D_{-1}\circ S_q=(1+(x-1)q)x\partial.
\]
By definition, $D_1$ commutes with $S_qT_qS_q$.

Our first result is that these three operators give a Lie algebra deforming $\mathfrak{sl}_2$:
\begin{Theorem}\label{Prop:deformed-sl2}
The operators $D_{-1}$, $D_0$ and $D_1$ form a Lie algebra with brackets
\begin{align*}
&[D_0,D_1]= \big(q^2-q+1\big)D_1+(1-q)D_0, \qquad
[D_0,D_{-1}] = -\big(q^2-q+1\big)D_{-1}+(1-q)D_0, \\
&[D_{-1},D_1]= 2D_0+(1-q)(D_{1}-D_{-1}).
\end{align*}
For $q=1$, we get the Lie algebra $\mathfrak{sl}_2$.
\end{Theorem}

\begin{proof}
The proof is a straightforward computation. All $D_i$ are of the form $g(x)\partial$ with $g$ a polynomial of degree at most 2. This explains why we can express any bracket as linear combination of $D_{-1}$, $D_0$ and $D_1$. The non-trivial part is that the coefficients are in $\Z[q]$.
Since~${D_0=(1+(x-1)q)D_{-1}}$, we get
\[[D_0,D_{-1}]=-D_{-1}(1+(x-1)q)D_{-1}=-q(1+(q-1)x)^2\partial.
\]
Similarly, we have $D_0=(1+(q-1)x)x^{-1}D_1$, hence
\[
[D_0,D_1]=-D_1\big(x^{-1}+q-1\big)D_1=(1+(x-1)q)^2\partial.
\]
The last bracket can be computed to be $[D_{-1},D_1]=\big(1-q+2qx+q(q-1)x^2\big)\partial$. One explicitely checks that these three brackets coincide with results claimed in the theorem.

Finally, it is clear that these brackets satisfy the Jacobi identity since we know a representation of the operators $D_i$ as differential operators.
\end{proof}

The Lie algebra $\mathfrak{sl}_2$ being simple, it does not allow any non-trivial deformations. Our $q$-deformation is indeed abstractly isomorphic to $\mathfrak{sl}_2$ when $q$ and $q^2-q+1$ are invertible. To give an explicit isomorphism, denote by $(f,h,e)$ the generators of $\mathfrak{sl}_2$ given by the differential operators $\big(\partial, x\partial, x^2\partial\big)$. They satisfy $[h,e]=e$, $[h,f]=-f$ and $[e,f]=-2h$. The following is an isomorphism of Lie algebras between $(D_{-1},D_0,D_1)$ and $(f,h,e)$:
\begin{gather*}
f=q^{-1/2}\left(D_{-1}+\frac{q-1}{q^2-q+1}D_0\right), \qquad h=\frac{D_0}{q^2-q+1}, \\
 e=q^{-1/2}\left(D_1+\frac{1-q}{q^2-q+1}D_0\right).
\end{gather*}

Using this isomorphism to $\mathfrak{sl}_2$, we can describe a 2-dimensional representation of the deformed Lie algebra defined by $(D_{-1}, D_0, D_1)$. Using the standard realization $f=\left(\begin{smallmatrix}0 & 0 \\1 & 0\end{smallmatrix}\right)$, \smash{$h=\left(\begin{smallmatrix}1/2 & 0 \\0 & -1/2\end{smallmatrix}\right)$} and $e=\left(\begin{smallmatrix}0 & -1 \\0 & 0\end{smallmatrix}\right)$, we get
\[
D_{-1}=\begin{pmatrix}\tfrac{1-q}{2}&0\vspace{1mm} \\q^{1/2} & \tfrac{q-1}{2}\end{pmatrix}, \qquad D_0=\begin{pmatrix}\tfrac{q^2-q+1}{2}&0\\0 & \tfrac{-q^2+q-1}{2}\end{pmatrix}, \qquad D_{1}=\begin{pmatrix}\tfrac{q-1}{2}& -q^{1/2}\\0 & \tfrac{1-q}{2}\end{pmatrix}.
\]

Note that this representation is not in $\mathfrak{sl}_2(\mathbb{Q}[q])$. A direct computation shows that there is no 2-dimensional representation of our $q$-deformed $\mathfrak{sl}_2$ into $\mathfrak{sl}_2(\mathbb{Q}[q])$. In dimension 3, there is of course the adjoint representation into $\mathfrak{sl}_3(\Z[q])$.

\begin{Remark}
It is tempting to consider $D_{-1}$, $D_1$ and $\widehat{D}_0:=[D_{-1},D_1]$. The operator $\widehat{D}_0$ still anti-commutes with $S_q$ and the bracket relations are
\[
[\widehat{D}_0,D_{\pm 1}]=\pm \big(q^2+1\big)D_{\pm 1}\pm (q-1)^2D_{\mp 1}.\]
The main drawback of this choice is that the eigenfunctions of $\widehat{D}_0$ are M\"obius transformations with coefficients not in $\Z[q]$.
\end{Remark}

\begin{Remark}\label{Rk:mod-square}
A simpler and very similar Lie algebra deforming $\mathfrak{sl}_2$ is given by generators $(d_{-1},d_0,d_1)$ with brackets
\begin{gather*}
[d_0,d_{-1}]= -qd_{-1} + (1 - q)d_0, \qquad [d_0,d_1]=qd_1 + (1 - q)d_0,\\ [d_{-1},d_{1}]=2d_0 + (1 - q)(d_{1} - d_{-1}).
\end{gather*}
It can be obtained as our deformation for a formal parameter $q$ with relation $(q-1)^2=0$. Then~${q^2-q+1=q}$. One checks that the Jacobi identity still holds.
\end{Remark}

\subsection[q-rational transition map]{$\boldsymbol{q}$-rational transition map}\label{Sec:Ovsienko}

The map $g_q$ defined in \eqref{Eq-def-g_q} plays a fundamental role, both for generalizing the $q$-deformation from $\mathfrak{sl}_2$ to the Witt algebra in Section \ref{Sec:Witt}, and in the theory of $q$-deformed rationals as we shall see now. It allows to pass between two different $q$-deformations of the rational numbers.

Recall that the $q$-rational transition map is defined by
\[
g_q(x)=\frac{1+(x-1)q}{1+(q-1)x},
\]
which is a deformation of the identity. It is the eigenfunction of $D_0$ with eigenvalue $q^2-q+1$ and normalization $g_q(0)=1-q$.
Note that $g_q$ is a M\"obius transformation associated to the matrix
\[
\begin{pmatrix} q & 1-q \\ q-1 & 1\end{pmatrix},
\]
which is of determinant $q^2-q+1$. For $q\neq 1$, \smash{$g_q$} is an elliptic transformation since its normalized trace is given by
\[
\frac{q+1}{\sqrt{q^2-q+1}}<2.
\] The unique fixed point on $\H^2$ is $\tfrac{1+{\rm i}\sqrt{3}}{2}$ which is independent of~$q$.

From the definition of $g_q$, we see the following duality between $q$ and $x$:
\[
g_q(x)g_x(q)=1.\]

\begin{Proposition}\label{Prop:qg-Tq-Sq}
The functions $g_q$, $T_q$ and $S_q$, seen as $2\times 2$ matrices satisfy:
\[
g_qT_q=qT_{q^{-1}}g_q \qquad \text{and}\qquad g_qS_q=qS_{q^{-1}}g_q.
\]
Therefore, seen as M\"obius transformations, we have $g_q\circ T_q = T_{q^{-1}}\circ g_q$ and $g_q\circ S_q = S_{q^{-1}}\circ g_q$.
\end{Proposition}
\begin{proof}
Both assertions can be checked by a direct computation:
\[
g_qT_q=\begin{pmatrix}q & 1-q\\q-1 & 1 \end{pmatrix}\begin{pmatrix}q & 1\\0 & 1 \end{pmatrix}=\begin{pmatrix}q^2 & 1\\q^2-q & q \end{pmatrix}=qT_{q^{-1}}g_q,
\]
and similarly
\[
g_qS_q=\begin{pmatrix}q & 1-q\\q-1 & 1 \end{pmatrix}\begin{pmatrix}0 & -1\\q & 0 \end{pmatrix}=\begin{pmatrix}q-q^2 & -q\\q & 1-q \end{pmatrix}=qS_{q^{-1}}g_q.
\]
The second identity can be derived also as follows: since $S_qD_0S_q=-D_0$, we see that both~$g_q^{-1}(x)$ and $g_q(S_q(x))$ are eigenfunctions of $D_0$ with eigenvalue $-q^2+q-1$. Hence they have to be multiple of each other. The precise relation is given by $g_q(S_q(x))=\tfrac{-q}{g_q(x)}=S_{q^{-1}}(g_q(x))$.
\end{proof}

We describe now the main link to $q$-deformed rational numbers. In \cite[Remark 3.2]{morier2022q}, the authors notice that the procedure for $q$-deformed irrational numbers gives two different answers when applied to rationals. This was further developed in \cite{bapat2023q}, from which we borrow the notations. When one approaches a rational $r/s$ from the right by a sequence of rationals strictly bigger than $r/s$, the procedure gives the so-called \emph{right $q$-rational} $[r/s]_q^\sharp$. This is the deformation obtained from applying $T_q$ and $S_q$ to zero described at the beginning of the Introduction. When approaching $r/s$ from the left, the limit gives another $q$-deformation of $r/s$, called \emph{left $q$-rational} and denoted by $[r/s]_q^\flat$ \cite[Theorem~2.11]{bapat2023q}.

The precise formulas given in \cite[Definition~2.6]{bapat2023q} can be written in our context as follows: consider $U=TST$, which is the function \smash{$U(x)=\tfrac{1}{1+1/x}$}, and its $q$-analog \smash{$U_q=T_qS_qT_q$}. For a~rational $r/s\in \Q$, take the unique even continued fraction expression $r/s=[a_1,a_2,\dots,a_{2n}]$. This means that $r/s=T^{a_1}U^{a_2}T^{a_3}\cdots U^{a_{2n}}(\infty)$. By convention, we put $\infty=[\,]$, the empty expression. Then
\begin{equation}\label{Eq:q-rat-sharp}
\left[\frac{r}{s}\right]_q^\sharp = T_q^{a_1}U_q^{a_2}T_q^{a_3}\cdots U_q^{a_{2n}}(\infty),
\end{equation}
and
\begin{equation}\label{Eq:q-rat-flat}
\left[\frac{r}{s}\right]_q^\flat = T_q^{a_1}U_q^{a_2}T_q^{a_3}\cdots U_q^{a_{2n}}\left(\frac{1}{1-q}\right).
\end{equation}

To give some examples, we have \smash{$[0]_q^\sharp=0$} and \smash{$[0]_q^\flat=\tfrac{q-1}{q}$}, \smash{$[1]_q^\sharp=1$} and \smash{$[1]_q^\flat=q$}, \smash{$[2]_q^\sharp=1+q$} and $[2]_q^\flat=1+q^2$, \smash{$[\infty]_q^\sharp=\infty$} and \smash{$[\infty]_q^\flat=\tfrac{1}{1-q}$}.

It was noticed numerically by Valentin Ovsienko that $g_q$ is a transition between these two~$q$-deformations of rational numbers. This is made precise in the following:
\begin{Theorem}\label{Thm:passage}
The passage between the two $q$-deformations of rationals is given by
\[
g_q\left(\left[\frac{r}{s}\right]_q^\sharp\right)=\left[\frac{r}{s}\right]_{q^{-1}}^\flat.
\]
\end{Theorem}
Note that $q$ gets inversed to $q^{-1}$. The proof is an application of Proposition \ref{Prop:qg-Tq-Sq}.
\begin{proof}
Proposition \ref{Prop:qg-Tq-Sq} gives $g_qU_q=U_{q^{-1}}g_q$. Using equation \eqref{Eq:q-rat-sharp} and again Proposition \ref{Prop:qg-Tq-Sq}, we get
\[
g_q\left(\left[\frac{r}{s}\right]_q^\sharp\right)=g_qT_q^{a_1}U_q^{a_2}T_q^{a_3}\cdots U_q^{a_{2n}}(\infty)=T_{q^{-1}}^{a_1}U_{q^{-1}}^{a_2}T_{q^{-1}}^{a_3}\cdots U_{q^{-1}}^{a_{2n}}g_q(\infty).
\]
Now $g_q(\infty)=\tfrac{q}{q-1}=\tfrac{1}{1-q^{-1}}$. Hence we conclude by equation \eqref{Eq:q-rat-flat}.
\end{proof}

As an application, we can reprove the positivity property of left $q$-rationals, proven in \cite[Appendix A.1]{bapat2023q} via an explicit combinatorial interpretation.

\begin{Corollary}
For $r/s>1$, we have $R^\flat, S^\flat\in\N[q]$, where $[r/s]_q^\flat=R^\flat(q)/S^\flat(q)$.
\end{Corollary}
\begin{proof}
From \cite[Proposition 1.3]{morier2020continued}, we know that for $r/s\in\Q_{>1}$ the right $q$-rational $[r/s]_q^\sharp=R^\sharp(q)/S^\sharp(q)$ is a rational function in $q$ with positive coefficients, i.e., $R^\sharp, S^\sharp\in \mathbb{N}[q]$. We also know from \cite[Theorem 2]{morier2020continued} that if $r/s> r'/s'$, then $R^\sharp S'^\sharp-R'^\sharp S^\sharp \in \N[q]$. Since ${r/s>1}$, we get ${R^\sharp-S^\sharp\in\N[q]}$. Since $r/s+1>r/s$ and $[r/s+1]_q^\sharp=q[r/s]_q^\sharp+1$, we also get ${(q-1)R^\sharp+S^\sharp \in\N[q]}$. Finally, we deduce that
\[
\left[\frac{r}{s}\right]_{q^{-1}}^\flat = g_q\left(\left[\frac{r}{s}\right]_q^\sharp\right)=\frac{q\big(R^\sharp-S^\sharp\big)+S^\sharp}{(q-1)R^\sharp+S^\sharp}
\]
has positive coefficients. Multiplying both numerator and denominator with an appropriate power of $q$, we get the same for $[r/s]^\flat_q$.
\end{proof}

Finally, we can use the transition function $g_q(x)$ as reparametrization of $\mathbb{P}^1$. To emphasize the dependence of our differential operators, we will write here $D_{-1}(q,x)=(1+(q-1)x)\partial_x$ and similar for $D_0$ and $D_{1}$.
\begin{Proposition}
Reparametrizing $\mathbb{P}^1$ by the transition map $\xi=g_q(x)$ gives
\begin{align*}
D_{\pm 1}(q,x) &= q D_{\pm 1}\big(q^{-1},\xi\big), \qquad
D_0(q,x) = \big(q^2-q+1\big)\xi \partial_\xi.
\end{align*}
\end{Proposition}
The behavior of $D_{-1}$ and $D_1$ is reminiscent of Proposition \ref{Prop:qg-Tq-Sq}.
\begin{proof}
Using \smash{$\frac{{\rm d}\xi}{{\rm d}x}=\tfrac{q^2-q+1}{(1+(q-1)x)^2}$} and \smash{$x=\tfrac{\xi+q-1}{q+(1-q)\xi}$}, we get \smash{$1+(q-1)x=\tfrac{q^2-q+1}{q+(1-q)\xi}$}. Hence
\[
D_{-1}(x,q)=(1+(q-1)x)\frac{{\rm d}\xi}{{\rm d}x}\partial_\xi = \frac{q^2-q+1}{1+(q-1)x}\partial_\xi = (q+(1-q)\xi)\partial_\xi = qD_{-1}\big(q^{-1},\xi\big).
\]
The computation for $D_{1}$ is similar. Finally,
\[
D_0(q,x)=(1+(q-1)x)(1+(x-1)q)\partial_x = \big(q^2-q+1\big)g_q(x)\partial_\xi = \big(q^2-q+1\big)\xi\partial_\xi.\tag*{\qed}
\]\renewcommand{\qed}{}
\end{proof}

This proposition indicates that we can use the undeformed operator $x\partial$ together with $D_{\pm 1}$ to get a deformation of $\mathfrak{sl}_2$ which is equivalent to our proposal. The importance of the $q$-rational transition map $g_q$, especially in the light of Proposition \ref{Prop:qg-Tq-Sq}, justifies to use $D_0$ instead of $x\partial$.

\subsection{Heisenberg algebra}
The operators $D_{-1}$ and $g_q$, seen as multiplication operator, give a deformation of the Heisenberg algebra. This strengthens the idea of considering $g_q$ as a deformation of the identity.

\begin{Theorem}
The two operators $D_{-1}$ and $g_q$ satisfy
\[
[D_{-1},g_q]=q+(1-q)g_q.
\]
Hence together with the central element $1$, they define a solvable $3$-dimensional Lie algebra deforming the $3$-dimensional Heisenberg algebra which we recover for $q=1$.
\end{Theorem}

The proof is a simple computation:
\[[D_{-1},g_q]=D_{-1}(g_q)=\frac{q^2-q+1}{1+(q-1)x}=q+(1-q)g_q.
\]
The Lie algebra generated by $(1,g_q,D_{-1})$ is solvable since $D_{-1}$ is not in the image of the Lie bracket. Hence the derived series becomes zero at the second step.

The previous theorem works since there is a nice expression for $D_{-1}(g_q)$. This holds true more generally:
\begin{Proposition}\label{Prop:op-diff-gq}
The function $g_q$ behaves well under the operators $D_{-1}$, $D_0$ and $D_1$:
\begin{align*}
D_0(g_q) &= \big(q^2-q+1\big)g_q, \qquad
D_{-1}(g_q) = q+(1-q)g_q, \qquad
D_1(g_q) = (q-1)g_q+g_q^2.
\end{align*}
\end{Proposition}
\begin{proof}
We only have to prove the last statement since we have already seen the first two. For that, we use the relation $g_q(S_q(x))=-q/g_q(x)$, see Proposition \ref{Prop:qg-Tq-Sq}. We get
\[
D_1(g_q)=S_qD_{-1}S_q(g_q)=S_qD_{-1}(-q/g_q)=S_q\bigl(-qg_q^{-2}(q+(1-q)g_q)\bigr),
\]
where we first used that $S_q$ acts by precomposition, then Proposition \ref{Prop:qg-Tq-Sq} and finally the expression for $D_{-1}(g_q)$. Since $S_q$ acts by precomposition applying \ref{Prop:qg-Tq-Sq} again concludes:
\[
D_1(g_q)S_q\bigl(-qg_q^{-2}(q+(1-q)g_q)\bigr)=g_q^2+(q-1)g_q.\tag*{\qed}
\]\renewcommand{\qed}{}
\end{proof}

We can use this proposition to express one operator in terms of another via the relation~\smash{$D_i=\tfrac{D_i(g_q)}{D_j(g_q)}D_j$}
for all $i,j\in \{-1,0,1\}$. This holds true since these differential operators are of order 1.

\section{Deformed Witt algebra}\label{Sec:Witt}

Now that we have deformed the differential operators $\partial$, $x\partial$ and $x^2\partial$, we can do the same for all~$x^n\partial$ for $n\in\Z$. These are a realization of the \emph{Witt algebra}, the Lie algebra of complex polynomial vector fields on the circle (the centerless Virasoro algebra). Putting $\ell_n=x^{n+1}\partial$, the Lie algebra structure is given by
\[\label{Eq-witt-algebra}
[\ell_n,\ell_m]=(m-n)\ell_{n+m}.
\]

To get a deformation of the Witt algebra, we define for $n>1$:
\begin{align*}
D_n &= g_q^{n-1}D_1, \qquad
D_{-n} = \big(qg_q^{-1}\big)^{n-1}D_{-1},
\end{align*}
where $g_q^{-1}=1/g_q$ denotes the inverse for multiplication (not composition).

\begin{Proposition}\label{Prop:mult-by-gq}
The operators $D_n$ behave nicely when multiplied by $g_q$. By definition we have~${g_qD_n=D_{n+1}}$ for $n\geq 1$ and $g_qD_{-n}=qD_{-n+1}$ for $n\geq 2$. In addition,
\begin{align*}
g_qD_0 &= (1-q)D_0+\big(q^2-q+1\big)D_1,\qquad
g_qD_{-1}= D_0+(1-q)D_1.
\end{align*}
Similarly, there is a nice behavior when multiplied by $qg_q^{-1}$. By definition $qg_q^{-1}D_{-n}=D_{-n-1}$ for~${n\geq 1}$ and $qg_q^{-1}D_n=qD_{n-1}$ for $n\geq 2$. In addition,
\begin{align*}
qg_q^{-1}D_0 &= (q-1)D_0+\big(q^2-q+1\big)D_{-1},\qquad
qg_q^{-1}D_{1}= D_0+(q-1)D_{-1}.
\end{align*}
\end{Proposition}
\begin{proof}
From the definitions, we get $g_qD_0=(1+(x-1)q)^2\partial$. From Proposition~\ref{Prop:deformed-sl2} and its proof, we see that this is $[D_0,D_1]$. Therefore, $g_qD_0=\big(q^2-q+1\big)D_1+(1-q)D_0$. A direct computation also gives $g_qD_{-1}=(1+(x-1)q)\partial=D_0+(1-q)D_1$.

For the second half, note that
\[
g_qD_0+(q-1)g_qD_{-1}=\big(q^2-q+1-(q-1)^2\big)D_1=qD_1.
\]
 Dividing by $g_q$ gives $qg_q^{-1}D_1=D_0+(q-1)D_{-1}$. Similarly,
 \[
 (q-1)g_qD_0+\big(q^2-q+1\big)D_{-1}=qD_0,\] so dividing by $g_q$ gives $qg_q^{-1}D_0=(q-1)D_0+\big(q^2-q+1\big)D_{-1}$.
\end{proof}

Using Propositions \ref{Prop:op-diff-gq} and \ref{Prop:mult-by-gq}, we get the bracket relations of all $D_n$:

\begin{Theorem}\label{Prop-deformed-Witt}
The $(D_n)_{n\in\mathbb{Z}}$ form a Lie algebra with bracket given by $($with $n,r >0)$:
\begin{gather*}
[D_0,D_n] = n\big(q^2-q+1\big)D_n+\big(q^2-q+1\big)\displaystyle\sum_{k=1}^{n-1}(1-q)^kD_{n-k} + (1-q)^nD_0, \\
[D_0,D_{-n}] = -n\big(q^2-q+1\big)D_{-n}-\big(q^2-q+1\big)\displaystyle\sum_{k=1}^{n-1}(q-1)^kD_{-n+k}-(q-1)^nD_0,\\
[D_n,D_{n+r}] =rD_{2n+r}+(q-1)rD_{2n+r-1},\\
[D_{-n},D_{-n-r}] =-rD_{-2n-r}+(q-1)rD_{-2n-r+1},\\
[D_{-n},D_n] = 2nq^{n-1}D_0+(2n-1)q^{n-1}(q-1)(D_{-1}-D_1),\\
[D_{n+r},D_{-n}] = (q-1)q^{n-1}(2n+r-1)D_{r+1}-\big(q^2+(2n+r-2)q+1\big)q^{n-1}D_r,\\
\phantom{[D_{n+r},D_{-n}] =}{} -q^{n-1}\big(q^2-q+1\big)\displaystyle\sum_{k=1}^{r-1}(1-q)^kD_{r-k}-(1-q)^rq^{n-1}D_0,\\
[D_{n},D_{-n-r}] = -(q-1)q^{n-1}(2n+r-1)D_{-r-1}-\big(q^2+(2n+r-2)q+1\big)q^{n-1}D_{-r},\\
\phantom{[D_{n},D_{-n-r}] =}{} -q^{n-1}\big(q^2-q+1\big)\displaystyle\sum_{k=1}^{r-1}(q-1)^kD_{-r+k}-(q-1)^rq^{n-1}D_0.
\end{gather*}
For $q=1$, one recovers the Witt algebra.
\end{Theorem}

It is clear that the bracket of the operators $D_n$ satisfies the Jacobi identity since these operators come from a realization as differential operators. We only have to check the bracket relations, which uses induction and all properties between $g_q$ and $D_{-1}$, $D_0$, $D_1$.

\begin{proof}
We prove the first relation by induction on $n$. The case $n=1$ is true by Proposition~\ref{Prop:deformed-sl2}. Then for $n>1$,
\begin{align*}
[D_0,D_n] ={}& [D_0,g_qD_{n-1}] = D_0(g_q)D_{n-1}+g_q[D_0,D_{n-1}]\\
={}& \big(q^2-q+1\big)g_qD_{n-1}+g_q(n-1)\big(q^2-q+1\big)D_{n-1} \\
& +g_q\left(\big(q^2-q+1\big)\displaystyle\sum_{k=1}^{n-2}(1-q)^kD_{n-1-k}+(1-q)^{n-1}D_0\right)\\
={}&n\big(q^2-q+1\big)D_n+\big(q^2-q+1\big)\displaystyle\sum_{k=1}^{n-1}(1-q)^kD_{n-k}+(1-q)^nD_0,
\end{align*}
where we used that $g_q$ is an eigenfunction of $D_0$, and Proposition \ref{Prop:mult-by-gq}.
The second statement is a similar computation.
The third relation comes as follows:
\begin{align*}
[D_n,D_{n+r}]&=[D_n,g_q^rD_n]=D_n(g_q^r)D_n=rg_q^{r-1}g_q^{n-1}D_1(g_q)D_n\\
&=rg_q^{r+n-2}\big(g_q^2+(q-1)g_q\big)D_n = rD_{2n+r}+r(q-1)D_{2n+r-1},
\end{align*}
where we used Proposition \ref{Prop:op-diff-gq} for $D_1(g_q)$. The fourth bracket is a similar computation.
To prove the fifth relation, we use induction on $n$ again. The initial $n=1$ is done by Proposition~\ref{Prop:deformed-sl2}. Then for $n\geq 1$,
\begin{align*}
[D_{-n-1},D_{n+1}]={}& \big[qg_q^{-1}D_{-n},g_qD_n\big]=qg_q^{-1}D_{-n}(g_q)D_n+q[D_{-n},D_n]-qg_qD_n\big(g_q^{-1}\big)D_{-n}\\
={}& q^ng_q^{-n}(q+(1- q)g_q)D_n+q\big(2nq^{n-1}D_0\\
&+(2n- 1)q^{n-1}(q-1)(D_{-1}-D_1)\big) +qg_q^{n-2}\big(g_q^2+(q-1)g_q\big)D_{-n}\\
={}& 2nq^nD_0+q^n(2n(q-1)+q_q)D_{-1}-q^n\big(2n(q-1)-qg_q^{-1}\big)D_1\\
={}&(2n+2)q^nD_0+(2n+1)q^n(q-1)(D_{-1}-D_1),
\end{align*}
where we used several times Propositions \ref{Prop:op-diff-gq} and \ref{Prop:mult-by-gq}. Finally, for the last two brackets, we start from (where $a,b>0$)
\begin{align}
[D_a,D_{-b}]&=\big[g_q^{a-1}D_1,\big(qg_q^{-1}\big)^{b-1}D_{-1}\big] \nonumber\\
&=q^{b-1}\big(g_q^{a-1}D_1\big(g_q^{1-b}\big)D_{-1}-g_q^{1-b}D_{-1}\big(g_q^{a-1}\big)D_1+g_q^{a-b}[D_1,D_{-1}]\big)\nonumber\\
&=q^{b-1}g_q^{a-b}(-(a+b)D_0+(a+b-1)(q-1)(D_1-D_{-1})).\label{Eq:aux.34}
\end{align}
An easy induction gives for $r>0$,
\begin{equation}\label{Eq:aux.56}
g_q^rD_0=\big(q^2-q+1\big) \sum_{k=0}^{r-1}(1-q)^kD_{r-k}+(1-q)^rD_0.
\end{equation}
Also we get $g_q^rD_{-1}=g_q^{r-1}D_0+(1-q)D_r$, where we can use equation \eqref{Eq:aux.56} to express the first term. Similar results hold for $\big(qg_q^{-1}\big)^rD_0$ and $\big(qg_q^{-1}\big)^rD_1$. Putting $a=n+r$ and $b=n$ in equation \eqref{Eq:aux.34} and using \eqref{Eq:aux.56} gives the bracket $[D_{n+r},D_{-n}]$. Putting $a=n$ and $b=n+r$ gives in a similar way the last bracket $[D_{n},D_{-n-r}]$.
\end{proof}

\begin{Remark}
Regarding Remark \ref{Rk:mod-square}, we could try to simplify the defining relations of the~$q$-deformed Witt algebra by considering a formal parameter $q$ satisfying $(q-1)^2=0$ (and then forget about this relation again). In contrast to the~$q$-deformed $\mathfrak{sl}_2$, the result here is not a Lie algebra anymore. The Jacobi identity does not hold exactly, but only modulo $(q-1)^2=0$.
\end{Remark}

\section{M\"obius transformations}\label{Sec:Moebius}

The differential operators $\partial$, $x\partial$, $x^2\partial$ can be interpreted in at least three different ways: first as differential operators on $\mathbb{P}^1$ written in one chart (this was our approach). Second they can be seen as complex vector fields on the circle $\mathbb{S}^1\subset \C$ (this approach was used for the Witt algebra). Third, a Lie algebra can be realised as Killing vector fields on the associated symmetric space of non-compact type. For $\mathfrak{sl}_2$ this symmetric space is the hyperbolic plane $\H^2$.

In this section, we integrate the operators $D_{-1}$, $D_0$ and $D_1$ seen as vector fields of $\H^2$. The result gives interesting M\"obius transformations with $q$-parameter. In the Taylor expansion around~${q\to 1}$ (the ``semi-classical limit''), we recover the deformed translation $T_q$. Conjecturally there should be a $q$-deformation of $\H^2$ on which these transformations act, such that the boundary can be identified with the $q$-deformed real numbers of \cite{morier2022q}.

\subsection{Classical setting}
Consider first the classical setup with the operators $\partial$, $x\partial$ and $x^2\partial$. These are Killing vector fields on the hyperbolic plane $\H^2$, whose integration determines isometries of $\H^2$. Here, we consider~${\H^2\subset \C}$ in the upper half-plane model and use the coordinate $x\in\C$.

To start, consider the case of the vector field $V=\partial=\tfrac{\partial}{\partial x}$. A curve $\gamma$ integrates this vector field iff $\gamma'(t)=V(\gamma(t))=1$ for all $t\in \R$ (where we identified 1 with the constant vector field~$\partial$). With initial condition $\gamma(0)=x$ we get $\gamma(t)=x+t$. We should think of this as a function $\gamma_x(t)$ of the initial condition. For time 1, we get the translation $\gamma_x(1)=x+1=T(x)$.

Another important case is $x\partial$, for which we have to solve $\gamma'(t)=\gamma(t)$. With initial condition~${\gamma(0)=x}$, we obviously get $\gamma(t)={\rm e}^t x$. The function $x\mapsto {\rm e}^t x$ is the hyperbolic isometry of~$\H^2$ associated to the geodesic joining 0 to $\infty$. Its matrix is given by
\begin{equation}\label{Eq:dilatation}
\begin{pmatrix} {\rm e}^{t/2} & 0 \\ 0 & {\rm e}^{-t/2} \end{pmatrix}.
\end{equation}

We can immediately generalize to the generators of the Witt algebra. Consider the operator~$x^n \partial$ with $n\in \mathbb{Z}$, $n\neq 1$. A curve $\gamma$ integrates the associated vector field if
$\gamma'(t)=\gamma(t)^{n}$. The solution with initial condition $\gamma(0)=x$ is given by
\[
\gamma_x(t)=\frac{x}{\big(1-(n-1)tx^{n-1}\big)^{1/(n-1)}}.
\]
Apart from $n=0$ and $n=2$, the associated transformations in $x$ are not M\"obius transformations. We get M\"obius transformations though when passing to a ramified covering. Putting $y=x^{n-1}$, we get
\[
\gamma_x(t)^{n-1}=\frac{y}{1-(n-1)ty}.
\]

\subsection{Deformed transformations}

We repeat the method of the previous subsection to deduce the transformations associated to~$D_{-1}$,~$D_0$ and $D_1$. Since these operators are still of the form $p(x)\partial$ with $p$ a polynomial in $x$ of degree at most 2, the vector fields $D_i$ are still Killing vector fields, so their integration gives M\"obius transformations.

Start with $D_{-1}\!=\!(1+(q-1)x)\partial$. The associated differential equation is ${\gamma'(t)=1+(q-1)\gamma(t)}$ with initial condition $\gamma(0)=x$. Solving this equation is standard: first one solves the homogeneous equation, then one uses the variation of the constant to finally get
\[
\gamma(t)=-\frac{1}{q-1}+\left(x+\frac{1}{q-1}\right){\rm e}^{(q-1)t}.
\]

For $t=1$, we get the associated map $x\mapsto -\tfrac{1}{q-1}+\big(x+\tfrac{1}{q-1}\big){\rm e}^{q-1}$. The Taylor expansion around $q-1$ at order 1 gives
\[
x\mapsto -\frac{1}{q-1}+\left(x+\frac{1}{q-1}\right)q = qx+1,
\]
which is nothing but $T_q(x)$.
For a general time $t$, the same procedure gives $x\mapsto (1-t+qt)x+t$. To sum up:
\begin{Proposition}
The time $1$ flow of the operator $D_{-1}$ seen as vector field on $\H^2$ is the affine map \smash{$x\mapsto -\tfrac{1}{q-1}+\big(x+\tfrac{1}{q-1}\big){\rm e}^{q-1}$} whose Taylor expansion at order $1$ in $q-1$ is $T_q$.
\end{Proposition}

For the operator $D_1$, it is not necessary to do any computation since $D_1=S_qD_{-1}S_q$. We can simply conjugate by $S_q$ the previous computations. In particular, the associated transformation in the Taylor expansion is $S_qT_qS_q$.

Consider now the operator $D_0=(1+(q-1)x)(1+(x-1)q)\partial$. The associated differential equation reads
\[
\gamma'(t)=1-q+\bigl(-1+3q-q^2\bigr)\gamma+q(q-1)\gamma^2,
\]
which is a Ricatti equation.

To solve a Ricatti equation, put $a=1-q$, $b=-1+3q-q^2$, $c=q(q-1)$ and introduce the new function $u$ such that $c\gamma(t)=-u'(t)/u(t)$. Then $u$ satisfies $u''(t)-bu'(t)+acu(t) = 0$. The discriminant has the nice expression $b^2-4ac=\big(q^2-q+1\big)^2$. The two roots of the characteristic equation are $q$ and $-(q-1)^2$. Hence we get $u(t)=C_1 {\rm e}^{qt}+C_2 {\rm e}^{-(q-1)^2t}$,
where $C_1$, $C_2$ are two constants. Since $\gamma=-u'/(cu)$, we can scale $C_1$ and $C_2$ by the same number without changing~$\gamma$. Putting $C_1=1-q$, we get
\[
\gamma(t)=\frac{{\rm e}^{qt}+C_2 \frac{q-1}{q}{\rm e}^{-(q-1)^2t}}{(1-q){\rm e}^{qt}+C_2 {\rm e}^{-(q-1)^2t}}.
\]
The initial condition $\gamma(0)=x$ gives
\[
C_2=\frac{q(1-q)x-q}{q-1-qx}.
\]
We already see that $\gamma_x(t)$ is a M\"obius transformation in $x$ since $C_2$ is. For time $t=1$, we get the following.

\begin{Proposition}
Integrating to time $t=1$ the operator $D_0$ seen as vector field in $\H^2$ gives the M\"obius transformation
\[
\gamma_x(1)=\frac{\big(q{\rm e}^q+(q-1)^2{\rm e}^{-(q-1)^2}\big)x+(1-q)\big({\rm e}^q-{\rm e}^{-(q-1)^2}\big)}{q(1-q)\big({\rm e}^q-{\rm e}^{-(q-1)^2}\big)x+(1-q)^2{\rm e}^q+q{\rm e}^{-(q-1)^2}}.
\]
\end{Proposition}

If we Taylor expand $\gamma_x(1)$ around $q-1$ to order 1, we get a quadratic polynomial in $x$. In order to keep a M\"obius transformation, we Taylor expand all entries of the associated $2\times 2$ matrix to order 1 in $q-1$. The result is
\[
W_q := \begin{pmatrix} e(1-2q) & (e-1)(q-1) \\ (e-1)(q-1) & -q\end{pmatrix},\]
where we used $q^2=2q-1$ coming from the Taylor expansion.
We see that $q=1$ gives the transformation $x\mapsto ex$.

A similar computation with arbitrary time $t$ gives
\[
W_q^t = \begin{pmatrix} {\rm e}^t(t-qt-q) & ({\rm e}^t-1)(q-1) \\ ({\rm e}^t-1)(q-1) & -q\end{pmatrix},
\]
which for $q=1$ gives the transformation $x\mapsto {\rm e}^t x$ from equation \eqref{Eq:dilatation}.

\subsection[Speculations about a q-deformed hyperbolic plane]{Speculations about a $\boldsymbol{q}$-deformed hyperbolic plane}

The above computations seem to indicate the existence of a $q$-deformed version of the hyperbolic plane $\H^2_q$ on which the transformations $T_q$, $S_q$, $g_q$ and $W_q$ act. A similar idea is developed in \cite{bapat2023q} where a compactification of the space of stability conditions for type $A_2$ is constructed.

The transformation $S_q(x)=-1/(qx)$ has only one fixed point given by $iq^{-1/2}$. This equals~$[i]_q$, the $q$-deformed version of $i$ from \cite[formula~(9)]{ovsienko2021towards}.
The translation $T_q(x)=qx+1$ has two fixed points at the (usual) boundary at infinity, given by $\infty$ and $1/(1-q)$. However, we expect the boundary of $\H^2_q$ to be $\R_q\cup \{\infty\}$, where $\R_q$ denotes the $q$-reals. On $\R_q$, the transformation $T_q$ has no fixed point since $T_q[x]_q=[x+1]_q$.

An important role should play the $q$-rational transition map $g_q(x)=\frac{1+(x-1)q}{1+(q-1)x}$. Since it deforms the identity, there are strictly more transformations in the deformed setting.
For~${q\neq 1}$, $g_q$~is an elliptic transformation with only fixed point on $\H^2$ given by \smash{$\frac{1+{\rm i}\sqrt{3}}{2}$} which is independent of $q$. In \cite[Part~2.3]{ovsienko2021towards}, it is shown that this complex number stays itself under $q$-deformation.
Note that both transformations $g_q$ and $T_qS_q$ are rotations around the same center. Hence they commute. Similarly, the matrix of $g_q^{-1}$ anti-commutes with the matrix of $S_q$.

These links between $q$-deformed numbers and the $q$-deformed $\mathfrak{sl}_2$-algebra are intriguing and might point towards a deeper relation.

\subsection*{Acknowledgements}

I warmly thank Valentin Ovsienko and Sophie Morier-Genoud for inspiration, many suggestions and fruitful exchanges, and Peter Smillie and Vladimir Fock for helpful discussions. I also thank the anonymous referees for their remarks improving the paper. I gratefully acknowledge support from the University of Heidelberg where this work has been carried out, in particular under ERC-Advanced Grant 101018839 and Deutsche Forschungsgemeinschaft (DFG, German Research Foundation) - Project-ID 281071066 - TRR 191.

\pdfbookmark[1]{References}{ref}
\LastPageEnding

\end{document}